\RequirePackage{fix-cm}
\documentclass[oneside,english]{amsart}
\usepackage[latin9]{inputenc}
\usepackage{babel}
\usepackage{amsthm}
\usepackage{amssymb}
\usepackage{fixltx2e}
\usepackage[unicode=true,
 bookmarks=true,bookmarksnumbered=false,bookmarksopen=false,
 breaklinks=false,pdfborder={0 0 1},backref=false,colorlinks=false]
 {hyperref}
\hypersetup{
 pdfauthor={Bachar Alhajjar}}

\makeatletter
\numberwithin{equation}{section}
\numberwithin{figure}{section}
\@ifundefined{lettrine}{\usepackage{lettrine}}{}
\theoremstyle{plain}
\newtheorem{thm}{\protect\theoremname}[section]
  \theoremstyle{definition}
  \newtheorem{defn}[thm]{\protect\definitionname}
  \theoremstyle{plain}
  \newtheorem{lem}[thm]{\protect\lemmaname}
  \theoremstyle{plain}
  \newtheorem{prop}[thm]{\protect\propositionname}
  \theoremstyle{plain}
  \newtheorem{cor}[thm]{\protect\corollaryname}

\theoremstyle{definition}

\makeatother

  \providecommand{\corollaryname}{Corollary}
  \providecommand{\definitionname}{Definition}
  \providecommand{\lemmaname}{Lemma}
  \providecommand{\propositionname}{Proposition}
\providecommand{\theoremname}{Theorem}

\begin{document}

\title{The Filtration Induced by a Locally Nilpotent Derivation}

\author{Bachar~~ALHAJJAR }
\begin{abstract}
We investigate the filtration corresponding to the degree function
induced by a non-zero locally nilpotent derivations $\partial$ and
its associated graded algebra. As an application we provide an efficient
method to recover the Makar-Limanov invariants, isomorphism classes
and automorphism groups of classically known algebras. We also present
a new class of examples which can be fully described with this method.
\end{abstract}

\address{Bachar ALHAJJAR, Institut de Math\'ematiques de Bourgogne, Universit\'e
de Bourgogne, 9 Avenue Alain Savary, BP 47870, 21078 Dijon, France.}

\email{bachar.alhajjar@gmail.com}

\maketitle

\section*{\textbf{\normalsize Introduction}}

Let $k$ be a field of characteristic zero, and let $B$ be a commutative
$k$-domain. A $k$-derivation $\partial\in\mathrm{Der}_{k}(B)$ is
said to be\emph{ locally nilpotent} if for every $a\in B$, there
is an integer $n\geq0$ such that $\partial^{n}(a)=0$. An important
invariant of $k$-domains admitting non-trivial locally nilpotent
derivations is the so called \emph{Makar-Limanov invariant} which
was defined by Makar-Limanov as the intersection $\mathrm{ML}(B)\subset B$
of kernels of all locally nilpotent derivations of $B$ (\cite{key-M-L1996}).
This invariant was initially introduced as a tool to distinguish certain
$k$-domains from polynomial rings but it has many other applications
for the study of $k$-algebras and their automorphism groups (\cite{key-ML2001}).
One of the main difficulties in applications is to compute this invariant
without a prior knowledge of all locally nilpotent derivations of
a given $k$-domain. 

In \cite{key-Kaliman Makar} S. Kaliman and L. Makar-Limanov developed
general techniques to determine the $\mathrm{ML}$-invariant for a
class of finitely generated $k$-domains $B$. The idea is to reduce
the problem to the study of homogeneous locally nilpotent derivations
on graded algebras $\mathrm{Gr}(B)$ associated to $B$. For this,
one considers appropriate filtrations $\mathcal{F}=\{\mathcal{F}_{i}\}_{i\in\mathbb{R}}$
on $B$ generated by so called real-valued weight degree functions
in such a way that every non-zero locally nilpotent derivation on
$B$ induces a non-zero homogeneous locally nilpotent derivation on
the associated graded algebra $\mathrm{Gr_{\mathcal{F}}}(B)$.

In particular, every $k$-domain $B$ admitting a non-zero locally
nilpotent derivation $\partial$ comes equipped with a natural filtration
by the $k$-sub-vector-spaces $\mathcal{F}_{i}=ker(\partial^{i+1})$,
$i\geq0$, that we call the \emph{$\partial$-filtration}. 

In this article we show that this filtration is convenient for the
computation of the $\mathrm{ML}$-invariant, and we give general methods
to describe the sub-spaces $ker(\partial^{i+1})$ and their associated
graded algebra.

Knowing this filtration gives a very precise understanding of the
structure of \emph{semi-rigid} $k$-domains, that is,$k$-domains
for which every locally nilpotent derivation gives rise to the same
filtration. For such rings the study of the $\partial$-filtration
is a very efficient tool to determine isomorphism types and automorphism
groups. We illustrate how the computation of $\mathrm{ML}$-invariant
of classically known semi-rigid $k$-domains can be simplified using
these types of filtration. We also present a new class of semi-rigid
$k$-domains which can be studied with this method.

\section{\textbf{Preliminaries}}

In this section we briefly recall basic facts on filtered algebra
and their relation with derivation in a form appropriate to our needs. 

In the sequel, unless otherwise specified $B$ will denote a commutative
domain over a field $k$ of characteristic zero. The set $\mathbb{Z}_{\geqslant0}$
of non-negative integers will be denoted by $\mathbb{N}$.

\subsection{\label{sub:Filtration-and-the} Filtration and associated graded
algebra}
\begin{defn}
\label{def.proper.filtration} An \textit{$\mathbb{N}$-filtration}
of $B$ is a collection $\{\mathcal{F}_{i}\}_{i\in\mathbb{N}}$ of
$k$-sub-vector-spaces of $B$ with the following properties: 

1- $\mathcal{F}_{i}\subset\mathcal{F}_{i+1}$ for all $i\in\mathbb{N}$
. 

2- $B=\cup_{i\in\mathbb{N}}\mathcal{F}_{i}$ . 

3- $\mathcal{F}_{i}.\mathcal{F}_{j}\subset\mathcal{F}_{i+j}$ for
all $i,j\in\mathbb{N}$ . 

\noindent  The filtration is called \emph{proper} if the following
additional property holds: 

4- If $a\in\mathcal{F}_{i}\setminus\mathcal{F}_{i-1}$ and $b\in\mathcal{F}_{j}\setminus\mathcal{F}_{j-1}$,
then $ab\in\mathcal{F}_{i+j}\setminus\mathcal{F}_{i+j-1}$.
\end{defn}
\noindent  There is a one-to-one correspondence between proper $\mathbb{N}$-filtrations
and so called $\mathbb{N}$-degree functions:
\begin{defn}
An $\mathbb{N}$\emph{-degree function} on $B$ is a map $\deg:B\longrightarrow\mathbb{N}\mathbb{\cup}\{-\infty\}$
such that, for all $a,b\in B$, the following conditions are satisfied: 

(1) $\deg(a)=-\infty$ $\Leftrightarrow$ $a=0$.

(2) $\deg(ab)=\deg(a)+\deg(b)$.

(3) $\deg(a+b)\leq max\{\deg(a),\deg(b)\}$.

\noindent  If the equality in (2) replaced by the inequality $\deg(ab)\leq\deg(a)+\deg(b)$,
we say that $\deg$ is an $\mathbb{N}$\emph{-semi-degree function.}
\end{defn}
Indeed, for an $\mathbb{N}$-degree function on $B$, the sub-sets
$\mathcal{F}_{i}=\{b\in B\mid\deg(b)\leq i\}$ are $k$-subvector
spaces of $B$ that give rise to a proper $\mathbb{N}$-filtration
$\{\mathcal{F}_{i}\}_{i\in\mathbb{N}}$. Conversely, every proper
$\mathbb{N}$-filtration $\{\mathcal{F}_{i}\}_{i\in\mathbb{N}}$,
yields an $\mathbb{N}$-degree function $\omega:B\longrightarrow\mathbb{N}\mathbb{\cup}\{-\infty\}$
defined by $\omega(0)=-\infty$ and $\omega(b)=i$ if $b\in\mathcal{F}_{i}\setminus\mathcal{F}_{i-1}$.
\begin{defn}
\label{Def.Gralg} Given a $k$-domain $B=\cup_{i\in\mathbb{N}}\mathcal{F}_{i}$
equipped with a proper $\mathbb{N}$-filtration, the associated graded
algebra $\mathrm{Gr}(B)$ is the $k$-vector space 
\[
\mathrm{Gr}(B)=\oplus_{i\in\mathbb{N}}\mathcal{F}_{i}/\mathcal{F}_{i-1}
\]
equipped with the unique multiplicative structure for which the product
of the elements $a+\mathcal{F}_{i-1}\in\mathcal{F}_{i}/\mathcal{F}_{i-1}$
and $b+\mathcal{F}_{j-1}\in\mathcal{F}_{j}/\mathcal{F}_{j-1}$, where
$a\in\mathcal{F}_{i}$ and $b\in\mathcal{F}_{j}$, is the element
\[
(a+\mathcal{F}_{i-1})(b+\mathcal{F}_{j-1}):=ab+\mathcal{F}_{i+j-1}\in\mathcal{F}_{i+j}/\mathcal{F}_{i+j-1}.
\]
Property 4 (proper) in Definition \ref{def.proper.filtration} ensures
that $\mathrm{Gr}(B)$ is a commutative $k$-domain when $B$ is an
integral domain. \noindent  Since for each $a\in B$ the set $\{n\in\mathbb{N}\mid a\in\mathcal{F}_{n}\}$
has a minimum, there exists $i$ such that $a\in\mathcal{F}_{i}$
and $a\notin\mathcal{F}_{i-1}$. So we can define a $k$-linear map
$\mathrm{gr}:B\longrightarrow\mathrm{Gr}(B)$ by sending $a$ to its
class in $\mathcal{F}_{i}/\mathcal{F}_{i-1}$, i.e $a\mapsto a+\mathcal{F}_{i-1}$,
and $\mathrm{gr}(0)=0$. We will frequently denote $\mathrm{gr}(a)$
simply by $\overline{a}$. Observe that $\mathrm{gr}(a)=0$ if and
only if $a=0$. 
\end{defn}
Denote by $\deg$ the $\mathbb{N}$-degree function $\deg:B\longrightarrow\mathbb{N}\cup\{-\infty\}$
corresponding to the proper $\mathbb{N}$-filtration $\{\mathcal{F}_{i}\}_{i\in\mathbb{N}}$.
We have the following properties.
\begin{lem}
\label{lem:graded relation} Given $a,b\in B$ the following holds:

\emph{P1)} $\overline{a\, b}=\overline{a}\,\overline{b}$, i.e. $\mathrm{gr}$
is a multiplicative map. 

\emph{P2) }If $\deg(a)>\deg(b)$, then $\overline{a+b}=\overline{a}$. 

\emph{P3)} If $\deg(a)=\deg(b)=\deg(a+b)$, then $\overline{a+b}=\overline{a}+\overline{b}$. 

\emph{P4)} If $\deg(a)=\deg(b)>\deg(a+b)$, then $\overline{a}+\overline{b}=0$,
in particular $\mathrm{gr}$ is not an additive map in general.\end{lem}
\begin{proof}
Let us assume that $\deg(a)=i$ and $\deg(b)=j$. By definition, $\deg(ab)=i+j$
means that $ab\in\mathcal{F}_{i+j}$ and $ab\notin\mathcal{F}_{i+j-1}$,
so $\overline{ab}=ab+\mathcal{F}_{i+j-1}:=(a+\mathcal{F}_{i-1})(b+\mathcal{F}_{j-1})=\overline{a}\,\overline{b}$.
Which gives P1. For P2 we observe that since $\deg(a+b)=\deg(a)$,
we have $\overline{a+b}=(a+b)+\mathcal{F}_{i-1}=(a+\mathcal{F}_{i-1})+(b+\mathcal{F}_{i-1})$,
and since $\mathcal{F}_{j-1}\subset\mathcal{F}_{j}\subseteq\mathcal{F}_{i-1}$
as $i>j$, we get $b+\mathcal{F}_{i-1}=0$. P3) is immediate, by definition.
Finally, assume by contradiction that $\overline{a}+\overline{b}\neq0$,
then $\overline{a}+\overline{b}=(a+\mathcal{F}_{i-1})+(b+\mathcal{F}_{i-1})=((a+b)+\mathcal{F}_{i-1})\neq0$,
which means that $a+b\notin\mathcal{F}_{i-1}$ and $\deg(a+b)=i$,
which is absurd. So P4 follows.
\end{proof}

\subsection{Derivations}

\indent\newline\noindent  By a \textit{$k$-derivation} of $B$,
we mean a $k$-linear map $D:B\longrightarrow B$ which satisfies
the Leibniz rule: For all $a,b\in B$; $D(ab)=aD(b)+bD(a)$. The set
of all $k$-derivations of $B$ is denoted by $\mathrm{Der}_{k}(B)$. 

\noindent  The \textit{kernel} of a derivation $D$ is the subalgebra
$\ker D=\left\{ b\in B;D(b)=0\right\} $ of $B$.

\noindent  The \textit{plinth ideal} of $D$ is the ideal $\mathrm{pl}(D)=\ker D\cap D(B)$
of $\mathrm{\ker}D$, where $D(B)$ denotes the image of $B$.

\noindent  An element $s\in B$ such that $D(s)\in\ker(D)\setminus\{0\}$
is called a \textit{local slice} for $D$.
\begin{defn}
Given a $k$-algebra $B=\cup_{i\in\mathbb{N}}\mathcal{F}_{i}$ equipped
with a proper $\mathbb{N}$-filtration, a $k$-derivation $D$ of
$B$ is said to respect the filtration if there exists an integer
$d$ such that $D(\mathcal{F}_{i})\subset\mathcal{F}_{i+d}$ for all
$i\in\mathbb{N}$. 

\noindent  If so, we define a $k$-linear map $\overline{D}:\mathrm{Gr}(B)\longrightarrow\mathrm{Gr}(B)$
as follows: If $D=0$, then $\overline{D}=0$ the zero map. Otherwise,
if $D\neq0$ then we let $d$ be the least integer such that $D(\mathcal{F}_{i})\subset\mathcal{F}_{i+d}$
for all $i\in\mathbb{N}$ and we define 
\[
\overline{D}:\mathcal{F}_{i}/\mathcal{F}_{i-1}\longrightarrow\mathcal{F}_{i+d}/\mathcal{F}_{i+d-1}
\]
by the rule $\overline{D}(a+\mathcal{F}_{i-1})=D(a)+\mathcal{F}_{i+d-1}$.
One checks that $\overline{D}$ satisfies the Leibniz rule, therefore
it is a $k$-derivation of the graded algebra $\mathrm{Gr}(B)$. Moreover
it is homogeneous of degree $d$, i.e $\overline{D}$ sends homogeneous
elements of degree $i$ to zero or to homogeneous elements of degree
$i+d$ . 
\end{defn}
Observe that $\overline{D}=0$ if and only if $D=0$. In addition,
$\mathrm{gr}(\ker D)\subset\ker\overline{D}$.

\section{$\mathrm{LND}$\textbf{-Filtrations and Associated Graded Algebras}}

In this section we introduce the $\partial$-filtration associated
with a locally nilpotent derivation $\partial$. We explain how to
compute this filtration and its associated graded algebra in certain
situation.
\begin{defn}
A $k$-derivation $\partial\in\mathrm{Der}_{k}(B)$ is said to be\emph{
locally nilpotent} if for every $a\in B$, there exists $n\in\mathbb{N}$
(depending of $a$) such that $\partial^{n}(a)=0$. The set of all
locally nilpotent derivations of $B$ is denoted by $\mathrm{\mathrm{LND}}(B)$.

\noindent  In particular, every locally nilpotent derivation $\partial$
of $B$ gives rise to a proper $\mathbb{N}$-filtration of $B$ by
the sub-spaces $\mathcal{F}_{i}=\mathrm{\ker}\partial^{i}$, $i\in\mathbb{N}$,
that we call the \emph{$\partial$-filtration. }It is straightforward
to check (see Prop. 1.9 in \cite{key-gene}) that the $\partial$-filtration
corresponds to the $\mathbb{N}$-degree function $\deg_{\partial}:B\longrightarrow\mathbb{N}\cup\{-\infty\}$
defined by 
\end{defn}
\begin{center}
$\deg_{\partial}(a):=min\{i\in\mathbb{N}\mid\partial^{i+1}(a)=0\}$,
and $\deg_{\partial}(0):=-\infty$. 
\par\end{center}

\noindent  Note that by definition $\mathcal{F}_{0}=\ker\partial$
and that $\mathcal{F}_{1}\setminus\mathcal{F}_{0}$ consists of all
local slices for $\partial$.\\
Let $\mathrm{Gr}_{\partial}(B)=\oplus_{i\in\mathbb{N}}\mathcal{F}_{i}/\mathcal{F}_{i-1}$
denote the associated graded algebra relative to the $\partial$-filtration
$\{\mathcal{F}_{i}\}_{i\in\mathbb{N}}$. Let $\mathrm{gr}_{\partial}:B\longrightarrow\mathrm{Gr_{\partial}}(B)$;
$a\overset{\mathrm{gr}_{\partial}}{\longmapsto}\overline{a}$ be the
natural map between $B$ and $\mathrm{Gr_{\partial}}(B)$ defined
in \ref{Def.Gralg}, where $\overline{a}$ denote $\mathrm{gr}_{\partial}(a)$.

The next Proposition, which is due to Daigle (Theorem 2.11 in \cite{key-gene}),
implies in particular that if $B$ is of finite transcendence degree
over $k$, then every non-zero $D\in\mathrm{LND}(B)$ respects the
$\partial$-filtration and therefore induces a non-zero homogeneous
locally nilpotent derivation $\overline{D}$ of $\mathrm{Gr}_{\partial}(B)$.
\begin{prop}
\label{Daigle} $($\textbf{\emph{Daigle}}$)$ Suppose that $B$ is
a commutative domain, of finite transcendence degree over $k$. Then
for every pair $D\in\mathrm{Der}_{k}(B)$ and $\partial\in\mathrm{\mathrm{LND}}(B)$,
$D$ respects the $\partial$-filtration. Consequently, $\overline{D}$
is a well defined homogeneous derivation of the integral domain $\mathrm{Gr}_{\partial}(B)$
relative to this filtration, and it is locally nilpotent if $D$ is
locally nilpotent.
\end{prop}

\subsection{Computing the $\partial$-filtration}

\indent\newline\noindent  Here, given a finitely generated $k$-domain
$B$, we describe a general method which enables the computation of
the $\partial$-filtration for a locally nilpotent derivation $\partial$
with finitely generated kernel. First we consider a more general situation
where the plinth ideal $\mathrm{pl}(\partial)$ is finitely generated
as an ideal in $\ker\partial$ then we deal with the case where $\ker\partial$
is itself finitely generated as a $k$-algebra.

Let $B=k[X_{1},\ldots,X_{n}]/I=k[x_{1},\ldots,x_{n}]$ be a finitely
generated $k$-domain, and let $\partial\in\mathrm{LND}(B)$ be such
that $\mathrm{pl}(\partial)$ is generated by precisely $m$ elements
$f_{1},\ldots,f_{m}$ as an ideal in $\ker\partial$. Denote by $\mathcal{F}=\{\mathcal{F}_{i}\}_{i\in\mathbb{N}}$
the $\partial$-filtration, then:

By definition $\mathcal{F}_{0}=\ker\partial$. Furthermore, given
elements $s_{i}\in\mathcal{F}_{1}$ such that $\partial(s_{i})=f_{i}$
for every $i\in\{1,\ldots,m\}$, it is straightforward to check that 

\textit{\emph{
\begin{eqnarray*}
\mathcal{F}_{1} & = & \mathcal{F}_{0}s_{1}+\ldots+\mathcal{F}_{0}s_{m}+\mathcal{F}_{0}.
\end{eqnarray*}
}}Letting $\deg_{\partial}(x_{i})=d_{i}$, we denote by $H_{j}$ the
$\mathcal{F}_{0}$-sub-module in $B$ generated by elements of degree
$j$ relative to $\deg_{\partial}$ of the form $s_{1}^{u_{1}}\ldots s_{m}^{u_{m}}x_{1}^{v_{1}}\ldots x_{n}^{v_{n}}$,
i.e., 
\[
H_{j}:=\sum_{\sum_{j}u_{j}+\sum_{i}d_{i}.v_{i}=j}\mathcal{F}_{0}\left(s_{1}^{u_{1}}\ldots s_{m}^{u_{m}}x_{1}^{v_{1}}\ldots x_{n}^{v_{n}}\right)
\]
 where $u_{j},v_{i}\in\mathbb{N}$ for all $i$ and $j$. The integer
$\sum_{j}u_{j}+\sum_{i}d_{i}v_{i}$ is nothing but $\deg_{\partial}(s_{1}^{u_{1}}.s_{2}^{u_{2}}\ldots.s_{m}^{u_{m}}.x_{1}^{v_{1}}.x_{2}^{v_{2}}.\ldots.x_{n}^{v_{n}})$.
Then we define a new $\mathbb{N}$-filration $\mathcal{G}=\{\mathcal{G}_{i}\}_{i\in\mathbb{N}}$
of $B$ by setting 
\[
\mathcal{G}_{i}=\sum_{j\leq i}H_{j}.
\]
By construction $\mathcal{G}_{i}\subseteq\mathcal{F}_{i}$ for all
$i\in\mathbb{N}$, with equality for $i=0,1$. The following result
provides a characterization of when these two filtrations coincide:
\begin{lem}
\label{lem:equal-if-it-is-proper} The filtrations $\mathcal{F}$
and $\mathcal{G}$ are equal if and only if $\mathcal{G}$ is proper.\end{lem}
\begin{proof}
One direction is clear since $\mathcal{F}$ is proper. Conversely,
suppose that $\mathcal{G}$ is proper with the corresponding $\mathbb{N}$-degree
function $\omega$ on $B$ (see \S \ref{sub:Filtration-and-the}).
Given $f\in\mathcal{F}_{i}\setminus\mathcal{F}_{i-1}$, $i>1$, for
every local slice $s\in\mathcal{F}_{1}\setminus\mathcal{F}_{0}$,
there\textit{\emph{ exist $f_{0}\neq0,a_{i}\neq0,a_{i-1},\ldots,a_{0}\in\mathcal{F}_{0}$
such that}} $f_{0}f=a_{i}s^{i}+a_{i-1}s^{i-1}+\cdots+a_{0}$ ( see
the proof of Lemma 4 in \cite{key-Makar}). Since $\omega(g)=0$ (resp.
$\omega(g)=1$) for every $g\in\mathcal{F}_{0}$ (resp. $g\in\mathcal{F}_{1}\setminus\mathcal{F}_{0}$),
we obtain 
\[
\omega(f)=\omega(f_{0}f)=\omega(a_{i}s^{i}+a_{i-1}s^{i-1}+\cdots+a_{0})=\max\{\omega(a_{i}s^{i})\}=i,
\]
 and so $f\in\mathcal{G}_{i}$.
\end{proof}
Next, we determine the $\partial$-filtration, for a locally nilpotent
derivation $\partial$ with finitely generated kernel, by giving an
effective criterion to decide when the $\mathbb{N}$-filtration $\mathcal{G}$
defined above is proper.

Hereafter, we assume that $0\in\mathrm{Spec}(B)$ and that $\ker(\partial)$
is generated by elements $z_{i}(x_{1},\ldots,x_{n})\in B$ such that
$z_{i}(0,\ldots,0)=0$, $i\in\{1,\ldots,r\}$, which is always possible
since $k\subset\ker\partial$. Since $\ker(\partial)$ is finitely
generated $k$-algebra, the plinth ideal \textit{\emph{$\mathrm{pl}(\partial)$
is finitely generated. So there exist $s_{1}(x_{1},\ldots,x_{n})$,
$\ldots$, $s_{m}(x_{1},\ldots,x_{n})$ $\in\mathcal{F}_{1}$ such
that }}$\mathcal{F}_{1}=\mathcal{F}_{0}s_{1}+\ldots+\mathcal{F}_{0}s_{m}+\mathcal{F}_{0}$.
We can also assume that $s_{i}(0,\ldots,0)=0$. 

Letting $J\subset k^{[r+n+m]}=k[Z_{1},\ldots,Z_{r}][X_{1},\ldots,X_{n}][S_{1},\ldots,S_{m}]$
be the ideal generated by $I$ and the elements $Z_{1}=z_{1}(X_{1},\ldots,X_{n}),\ldots,Z_{r}=z_{r}(x_{1},\ldots,x_{n})$,
$S_{1}=s_{1}(X_{1},\ldots,X_{n}),\ldots,S_{m}=s_{m}(x_{1},\ldots,x_{n})$,
then we have 
\[
B=k[Z_{1},\ldots,Z_{r}][X_{1},\ldots,X_{n}][S_{1},\ldots,S_{m}]/J.
\]
Note that by construction $\underset{r+n+m\mathrm{\, times}}{(0,\ldots,0)}\in\mathrm{Spec}(k^{[r+n+m]}/J)$. 

We define an $\mathbb{N}$-degree function $\omega$ on $k^{[r+n+m]}$
by declaring that $\omega(Z_{i})=0=\deg_{\partial}(z_{i})$ for all
$i\in\{1,\ldots,r\}$, $\omega(S_{i})=1=\deg_{\partial}(s_{i})$ for
all $i\in\{1,\ldots,m\}$, and $\omega(X_{i})=\deg_{\partial}(x_{i})=d_{i}$
for all $i\in\{1,\ldots,n\}$. The corresponding proper $\mathbb{N}$-filtration
$\mathcal{Q}_{i}:=\{P\in k^{[n]}\mid\omega(P)\leq i\}$, $i\in\mathbb{N}$,
on $k^{[r+n+m]}$ has the form $\mathcal{\mathcal{Q}}_{i}=\oplus_{j\leq i}\mathcal{H}_{j}$
where 
\[
\mathcal{H}_{j}:=\oplus_{\sum_{j}u_{j}+\sum_{i}d_{i}v_{i}=j}k[Z_{1},\ldots,Z_{r}]S_{1}^{u_{1}}\ldots S_{m}^{u_{m}}X_{1}^{v_{1}}\ldots X_{n}^{v_{n}}.
\]
 By construction $\pi\left(\mathcal{\mathcal{Q}}_{i}\right)=\mathcal{G}_{i}$
where $\pi:k^{[r+n+m]}\longrightarrow B$ denotes the natural projection.
Indeed, since 
\[
\pi\left(\mathcal{\mathcal{Q}}_{i}\right)=\sum_{j\leq i}\pi\left(\mathcal{H}_{j}\right)
\]
 and $\pi\left(\mathcal{H}_{j}\right)=\sum_{\sum_{j}u_{j}+\sum_{i}d_{i}v_{i}=j}k[z_{1},\ldots,z_{r}]s_{1}^{u_{1}}\ldots.s_{m}^{u_{m}}x_{1}^{v_{1}}\ldots x_{n}^{v_{n}}$,
we get 
\[
\pi\left(\mathcal{H}_{j}\right)=\sum_{\sum_{j}u_{j}+\sum_{i}d_{i}v_{i}=j}\left(\ker\partial\right)s_{1}^{u_{1}}.\ldots s_{m}^{u_{m}}x_{1}^{v_{1}}.\ldots x_{n}^{v_{n}}
\]
 which means precisely that $\pi\left(\mathcal{\mathcal{Q}}_{i}\right)=\mathcal{G}_{i}$.

Let $\hat{J}\subset k^{[r+n+m]}$ be the homogeneous ideal generated
by the highest homogeneous components relative to $\omega$ of all
elements in $J$. Then we have the following result, which is inspired
by the technique developed by S. Kaliman and L. Makar-Limanov:
\begin{prop}
\label{Pro:when-it-is-proper} The $\mathbb{N}$-filration $\mathcal{G}$
is proper if and only if $\hat{J}$ is prime.\end{prop}
\begin{proof}
It is enough to show that $\mathcal{G}=\{\pi\left(\mathcal{\mathcal{Q}}_{i}\right)\}_{i\in\mathbb{N}}$
coincides with the filtration corresponding to the $\mathbb{N}$-semi-degree
function $\omega_{B}$ on $B$ defined by $\omega_{B}(p):=\min_{P\in\pi^{-1}(p)}\{\omega(P)\}$.
Indeed, if so, the result will follow from Lemma 3.2 in \cite{key-Kaliman Makar}
which asserts in particular that $\omega_{B}$ is an $\mathbb{N}$-degree
function on $B$ if and only if $\hat{J}$ is prime. Let $\{\mathcal{G}_{i}^{'}\}_{i\in\mathbb{N}}$
be the filtration corresponding to $\omega_{B}$. Given $f\in\mathcal{G}_{i}^{'}$
there exists $F\in\mathcal{\mathcal{Q}}_{i}$ such that $\pi(F)=f$,
which means that $\mathcal{G}_{i}^{'}\subset\pi\left(\mathcal{\mathcal{Q}}_{i}\right)$.
Conversely, it is clear that $\omega_{B}(z_{i})=\omega(Z_{i})=0$
for all $i\in\{1,\ldots,r\}$. Furthermore $\omega_{B}(s_{i})=\omega(S_{i})=1$
for all $i\in\{1,\ldots,m\}$, for otherwise $s_{i}\in\ker\partial$
which is absurd. Finally, if $d_{i}\neq0$ and $\omega_{B}(x_{i})<\omega(X_{i})=d_{i}\neq0$,
then $x_{i}\in\pi\left(\mathcal{\mathcal{Q}}_{d_{i}-1}\right)\subset\ker\partial^{d_{i}-1}$
which implies that $\deg_{\partial}(x_{i})<d_{i}$, a contradiction.
So $\omega_{B}(x_{i})=d_{i}$. Thus $\omega_{B}(f)\leq i$ for every
$f\in\pi\left(\mathcal{\mathcal{Q}}_{i}\right)$ which means that
$\pi\left(\mathcal{\mathcal{Q}}_{i}\right)\subset\mathcal{G}_{i}^{'}$.
\end{proof}
The next Proposition, which is a reinterpretation of Prop. 4.1 in
\cite{key-Kaliman Makar}, describes in particular the associated
graded algebra $Gr_{\partial}(B)$ of the filtered algebra $\left(B,\mathcal{F}\right)$
in the case where the $\mathbb{N}$-filtration $\mathcal{G}$ is proper:
\begin{prop}
\label{prop:graded algebra}If the $\mathbb{N}$-filtration $\mathcal{G}$\textup{
}\textup{\emph{is proper then }}\textup{$Gr_{\partial}(B)\simeq k^{[r+n+m]}/\hat{J}$.} \end{prop}
\begin{proof}
By virtue of $($ Prop. 4.1 in \cite{key-Kaliman Makar}$)$ the graded
algebra associated to the filtered algebra $\left(B,\mathcal{G}\right)$
is isomorphic to $k^{[r+n+m]}/\hat{J}$. So the assertion follows
from Lemma \ref{lem:equal-if-it-is-proper}. 
\end{proof}

\section{\textbf{Semi-Rigid and Rigid $k$-Domains}}

In \cite{key-David Maubach} D. Finston and S. Maubach considered
rings $B$ whose sets of locally nilpotent derivations are \textquotedblleft{}one-dimensional\textquotedblright{}
in the sense that $\mathrm{LND}(B)=\ker(\partial).\partial$ for some
non-zero $\partial\in\mathrm{LND}(B)$. They called them almost-rigid
rings. Hereafter, we consider the following definition which seems
more natural in our context (see Prop. \ref{Pro:alri-rigid one dimensional}
below for a comparison between the two notions). 
\begin{defn}
\label{def almost-rigid} A commutative domain $B$ over a field $k$
of characteristic zero is called \emph{semi-rigid} if all non-zero
locally nilpotent derivations of $B$ induce the same \textit{\emph{proper
$\mathbb{N}$-filtration (equivalently, the same $\mathbb{N}$-degree
function).}} 
\end{defn}
Recall that the\emph{ Makar-Limanov invariant} of a commutative $k$-domain
$B$ over a field $k$ of characteristic zero is defined by 
\[
\mathrm{ML}(B):=\cap_{D\in\mathrm{\mathrm{LND}}(B)}\ker(D).
\]
In particular, $B$ is semi-rigid if and only if $\mathrm{ML}(B)=\ker(\partial)$
for any non-zero $\partial\in\mathrm{LND}(B)$. Indeed, given $D,E\in\mathrm{LND}(B)\setminus\{0\}$
such that $A:=\ker(D)=\ker(E)$, there exist non-zero elements $a,b\in A$
such that $aD=bE$ (see \cite{key-gene} Principle 12) which implies
that the $D$-filtration is equal to the $E$-filtration. So if $\mathrm{ML}(B)=\ker(\partial)$
for any non-zero $\partial\in\mathrm{LND}(B)$ then $B$ is semi-rigid.
The other implication is clear by definition. 

Recall that $D\in\mathrm{Der}_{k}(B)$ is \textit{irreducible} if
and only if $D(B)$ is contained in no proper principal ideal of $B$,
and that $B$ is said to satisfy the ascending chain condition (ACC)
on principal ideals if and only if every infinite chain $(b_{1})\subset(b_{2})\subset(b_{3})\subset\cdots$
of principal ideals of $B$ eventually stabilizes. $B$ is said to
be a \textit{highest common factor ring}, or HCF-ring, if and only
if the intersection of any two principal ideals of $B$ is again principal.
\begin{prop}
\label{Pro:alri-rigid one dimensional} Let $B$ be a semi-rigid $k$-domain
satisfying the ACC on principal ideals. If $\mathrm{ML}(B)$ is an
HCF-ring, then there exists a unique irreducible $\partial\in\mathrm{LND}(B)$
up to multiplication by unit.\end{prop}
\begin{proof}
Existence: since $B$ is satisfies the ACC on principal ideals, then
for every non-zero $T\in\mathrm{LND}(B)$, there exists an irreducible
$T_{0}\in\mathrm{LND}(B)$ and $c\in\ker(T)$ such that $T=cT_{0}$.
(see \cite{key-gene}, Prop. 2.2 and Principle 7). 

Uniqueness: the following argument is similar to that in \cite{key-gene}
Prop. 2.2.b, but with a little difference, that is, in \cite{key-gene}
it is assumed that $B$ itself is an HCF-ring while here we only require
that $\mathrm{ML}(B)$ is an HCF-ring.

Let $D,E\in\mathrm{LND}(B)$ be irreducible derivations, and denote
$A=\mathrm{ML}(B)$. By hypothesis $\ker(D)=\ker(E)=A$, so there
exist non-zero $a,b\in A$ such that $aD=bE$ (see \cite{key-gene}
Principle 12). Here we can assume that $a,b$ are not units otherwise
we are done. Set $T=aD=bE$. Since $A$ is an HCF-ring, there exists
$c\in A$ such that $aA\cap bA=cA$. Therefore, $T(B)\subset cB$,
and there exists $T_{0}\in\mathrm{LND}(B)$ such that $T=cT_{0}$.
Write $c=as=bt$ for $s,t\in B$. Then $cT_{0}=asT_{0}=aD$ implies
$D=sT_{0}$, and likewise $E=tT_{0}$. By irreducibility, $s$ and
$t$ are units of $B$, and we are done.
\end{proof}
In particular, for a ring $B$ as in \ref{Pro:alri-rigid one dimensional},
every $D\in\mathrm{LND}(B)$ has the form $D=f\partial$ for some
irreducible $\partial\in\mathrm{LND}(B)$ and $f\in\ker(\partial)$,
and so $B$ is almost rigid in the sense introduced by Finston and
Maubach.

Recall that a ring $A$ is called \textit{rigid} if the zero derivation
is the only locally nilpotent derivation of $A$. Equivalently, $A$
is rigid if and only if $\mathrm{ML}(A)=A$. It is well known that
the only non rigid $k$-domains of transcendence degree one are polynomial
rings in one variable over algebraic extensions of $k$ (\cite{key-gene}
Corollary 1.24 and Corollary 1.29). In particular, we have the following
elementary criterion for rigidity that we will use frequently in the
sequel:
\begin{lem}
\label{rem rigid example} A domain $B$ of transcendence degree one
over a field $k$ of characteristic zero is rigid if one of the following
properties hold:

\emph{(1)} \textbf{$B$} is not factorial.

\emph{(2)} $\mathrm{Spec}(B)$ has a singular point.
\end{lem}

\subsection{Elementary examples of semi-rigid $k$-domains}

\indent\newline\noindent  The next Proposition, which is due to Makar-Limanov
(\cite{key-Makar} Lemma 21, also \cite{key-Crachiola Makar} Theorem
3.1), presents some of the simplest examples of semi-rigid $k$-domains.
\begin{prop}
\textbf{\emph{$($Makar-Limanov$)$}} Let $A$ be a rigid domain of
finite transcendence degree over a field $k$ of characteristic zero.
Then the polynomial ring $A[x]$ is semi-rigid.\end{prop}
\begin{proof}
For the convenience of the reader, we provide an argument formulated
in the $\mathrm{LND}$-filtration language. Let $\partial$ be the
locally nilpotent derivation of $A[x]$ defined by $\partial(a)=0$
for every $a\in A$ and $\partial(x)=1$. Then the $\partial$-filtration
$\{\mathcal{F}_{i}\}_{i\in\mathbb{N}}$ is given by $\mathcal{F}_{i}=Ax^{i}\oplus\mathcal{F}_{i-1}$
where $\mathcal{F}_{0}=\ker(\partial)=A$. So the associated graded
algebra is $Gr(A[x])=\oplus_{i\in\mathbb{N}}A\overline{x}^{i}$, where
$\overline{x}:=gr(x)$. By Proposition \ref{Daigle}, every non-zero
$D\in\mathrm{LND}(A[x])$ respects the $\partial$-filtration and
induces a non-zero homogeneous locally nilpotent derivation $\overline{D}$
of $\mathrm{Gr}(A[x])$ of a certain degree $d$. 

Let $f\in\ker(D)$ and assume that $f\notin A$. Then $\overline{x}$
divides $\overline{f}$. Since $\overline{f}\in\ker(\overline{D})$
and $\ker(\overline{D})$ is factorially closed (\cite{key-gene}
Principle 1), we have $\overline{x}\in\ker(\overline{D})$. Note that
$\overline{D}$ sends homogeneous elements of degree $i$ to zero
or to homogeneous elements of degree $i+d$, therefore $\overline{D}$
has the form $\overline{D}=a\overline{x}^{d}E$, where $a\in A$,
$d\in\mathbb{N}$, and $E\in\mathrm{LND}(A)$ see \cite{key-gene}
Principle 7. But since $A$ is rigid, $E=0$. Thus $\overline{D}=0$,
a contradiction. This means $f\in A$ which implies that $\ker(D)\subset A$.
Finally, since $\mathrm{tr.deg_{k}}(A)=\mathrm{tr.deg_{k}}(\ker(D))$
and $A$ is algebraically closed in $A[x]$, we get the equality $\ker(D)=A$.
Hence $\mathrm{ML}(A[x])=A$. 
\end{proof}

\section{\textbf{Computing the $\mathrm{ML}$-Invariant using $\mathrm{LND}$-Filtration }}

Here we illustrate the use of the $\partial$-filtration in the computation
of $\mathrm{ML}$-invariants, for classes of already well-studied
examples.

\subsection{Classical examples of semi-rigid $k$-domains}

\indent\newline\noindent  We first consider in \ref{ex:Daniel},
certain surfaces in $\mathbb{A}^{[3]}$ defined by equations $X^{n}Z-P(Y)$
where $n>1$ and $\deg_{Y}P(Y)>1$, which were first discussed by
Makar-Limanov in \cite{key-Makar}, where he computed their $\mathrm{ML}$-invariants.
Later on Poloni \cite{key-Poloni} used similar methods to compute
$\mathrm{ML}$-invariants for a larger class. In the second example
\ref{sub:The-second-type}, we consider certain threefolds whose invariants
were computed by S. Kaliman and L. Makar-Limanov \cite{key-sh and makar}
in the context of the linearization problem for $\mathbb{C}^{*}$-action
on $\mathbb{C}^{3}$. 

In these examples, the use of $\mathrm{LND}$-filtrations is more
natural and less tedious than other existing approaches.

\subsubsection{\label{ex:Daniel} \textbf{\emph{Danielewski hypersurfaces}}}

\indent\newline\noindent  Let 
\[
B_{n,P}=k[X,Y,Z]/\langle X^{n}Z-P(X,Y)\rangle
\]
 where 
\[
P(X,Y)=Y^{m}+f_{m-1}(X)Y^{m-1}+\cdots+f_{0}(X),
\]
 $f_{i}(X)\in k[X]$, $n\geq2$, and $m\geq2$. 

Let $x$, $y$, $z$ be the images of $X$, $Y$, $Z$ in $B_{n,P}$.
Define $\partial$ by $\partial(x)=0$, $\partial(y)=x^{n}$, $\partial(z)=\frac{\partial P}{\partial y}$
where 
\[
\frac{\partial P}{\partial y}=my^{m-1}+(m-1)f_{m-1}(x)y^{m-2}+\ldots+f_{1}(x)
\]
 We see that $\partial\in\mathrm{LND}(B_{n,P})$ with $\ker(\partial)=k[x]$,
and $y$ is a local slice for $\partial$. Moreover, we have $\deg_{\partial}(x)=0$,
$\deg_{\partial}(y)=1$, $\deg_{\partial}(z)=m$. The plinth ideal
is $\mathrm{pl}(\partial)=\left\langle x^{n}\right\rangle $. Up to
a change of variable of the form $Y\mapsto Y-c$ where $c\in k$,
we can always assume that $0\in\mathrm{Spec}(B)$. A consequence of
Lemma \ref{lem:equal-if-it-is-proper}, Proposition \ref{Pro:when-it-is-proper},
and Proposition \ref{prop:graded algebra} (see the Proof of Prop.
\ref{Prop:main-result-filtration} for more details) is that 

1- The $\partial$-filtration $\{\mathcal{F}_{i}\}_{i\in\mathbb{N}}$
is given by: 
\[
\mathcal{F}_{mi+j}=k[x]y^{j}z^{i}+\mathcal{F}_{mi+j-1}
\]
where $i\in\mathbb{N}$ and $j\in\{0,\ldots m-1\}$.

2- The associated graded algebra $\mathrm{Gr}(B_{n,P})=\oplus_{i\in\mathbb{N}}\overline{B}_{i}$,
where $\overline{B}_{i}=\mathcal{F}_{i}/\mathcal{F}_{i-1}$, is generated
by $\overline{x}=gr_{\partial}(x)$, $\overline{y}=gr_{\partial}(y)$,
$\overline{z}=gr_{\partial}(z)$ as an algebra over $k$ with relation
$\overline{x}^{n}\overline{z}=\overline{y}^{m}$, i.e. $\mathrm{Gr}(B_{n,P})=k[\overline{X},\overline{Y},\overline{Z}]/\langle\overline{X}^{n}\overline{Z}-\overline{Y}^{m}\rangle$.
And we have : 
\[
\overline{B}_{mi+j}=k[\overline{x}]\overline{y}^{j}\overline{z}^{i}
\]
where $i\in\mathbb{N}$ and $j\in\{0,\ldots m-1\}$.
\begin{prop}
\label{prop:Daniel-example} With the notation above we have: 

\emph{(1)} $\mathrm{ML}(B_{n,P})=k[x]$. Consequently $B_{n,P}$ is
semi-rigid.

\emph{(2)} Every $D\in\mathrm{LND}(B_{n,P})$ has the form $D=f(x)\partial$.\end{prop}
\begin{proof}
(1) By Proposition \ref{Daigle} a non-zero $D\in\mathrm{LND}(B)$
induces a non-zero $\overline{D}\in\mathrm{LND}\left(\mathrm{Gr}(B)\right)$.
Let $f\in\ker(D)\setminus k$, then $\overline{f}\in\ker(\overline{D})\setminus k$.
There exists $i\in\mathbb{N}$ such that $\overline{f}\in\overline{B}_{i}$. 

Assume that $\overline{f}\notin k[\overline{x}]=\overline{B}_{0}$,
then one of the elements $\overline{y},\overline{z}$ must divide
$\overline{f}$ . Which leads to a contradiction as follows: 

If $\overline{y}$ divides $\overline{f}$ , then $\overline{y}\in\ker(\overline{D})$
because $\ker(\overline{D})$ is factorially closed in $\mathrm{Gr}(B)$
(\cite{key-gene} Principle 1). For the same reason $\overline{x},\overline{z}\in\ker(\overline{D})$
because $\overline{x}^{n}\overline{z}=\overline{y}^{m}$. So $\overline{D}=0$,
a contradiction. 

If $\overline{z}$ divides $\overline{f}$, then $\overline{D}(\overline{z})=0$.
So $\overline{D}$ extends to a locally nilpotent derivation $\widetilde{D}$
of the ring $\widetilde{B}=k(\overline{z})[\overline{x},\overline{y}]/\left\langle \overline{x}^{n}\overline{z}-\overline{y}^{m}\right\rangle $.
Since $0\in\mathrm{Spec}(B)$ is a singular point when $n\geq2$ and
$m\geq2$ , $\widetilde{B}$ is rigid (Lemma \ref{rem rigid example})\textit{.}
Therefore, $\widetilde{D}=0$ which means $\overline{D}=0$, a contradiction. 

So the only possibility is that $\overline{f}\in k[\overline{x}]$.
This means that $\deg_{\partial}(f)=0$, and hence that $f\in k[x]$.
So $\ker(D)\subset k[x]$, and finally $k[x]=\ker(D)$ because $\mathrm{tr.deg_{k}}(\ker(D))=1$
and $k[x]$ is algebraically closed in $B$. So we get $\mathrm{ML}(B)=k[x]$.

(2) is an immediate consequence of Proposition \ref{Pro:alri-rigid one dimensional}.
\end{proof}

\subsubsection{\label{sub:The-second-type}\textbf{\emph{ Koras-Russell hypersurfaces
of the second type}}}

\indent\newline\noindent  Here we consider hypersurfaces associated
with $k$-algebras of the form:

\[
B_{n,e,l,Q}=k[X,Y,Z,T]/\langle Y(X^{n}+Z^{e})^{l}-Q(X,Z,T\rangle
\]
 where 
\[
Q(X,Z,T)=T^{m}+f_{1}(X,Z)T^{m-1}+\ldots+f_{m}(X,Z)
\]
 $f_{i}(X,Z)\in k[X,Z]$, $n>1$, $e>1$, $l>1$, and $m>1$. We may
assume without loss of generality that $Q(0,0,0)=0$ . A particular
case of this family corresponds to the so called Koras-Russell hypersurfaces
of the second type considered by S. Kaliman and L. Makar-Limanov (\cite{key-sh and makar})
where they computed their $\mathrm{ML}$-invariants. Here we explain
how to apply the $\mathrm{LND}$-filtration method to compute this
invariant for all algebras $B_{n,e,l,Q}$.

Let $x$, $y$, $z$, $t$ be the images of $X$, $Y$, $Z$, $T$
in $B_{n,e,l,Q}$. Define $\partial$ by 
\[
\partial=\frac{\partial Q}{\partial t}\partial_{y}+(X^{n}+Z^{e})^{l}\partial_{t}
\]
 We see that $\partial\in\mathrm{LND}(B_{n,e,l,Q})$ with $\ker(\partial)=k[x,z]$,
and $t$ is a local slice for $\partial$. Moreover, we have $\deg_{\partial}(x)=0$,
$\deg_{\partial}(y)=m$, $\deg_{\partial}(z)=0$, and $\deg_{\partial}(t)=1$.
The plinth ideal is $\mathrm{pl}(\partial)=\langle(X^{n}+Z^{e})^{l}\rangle$.
By Lemma \ref{lem:equal-if-it-is-proper}, Prop. \ref{Pro:when-it-is-proper},
and Prop. \ref{prop:graded algebra} we get the following. 

1- The $\partial$-filtration $\{\mathcal{F}_{i}\}_{i\in\mathbb{N}}$
is given by: 
\[
\mathcal{F}_{mi+j}=k[x,z]t^{j}y^{i}+\mathcal{F}_{mi+j-1}
\]
where $i\in\mathbb{N}$, and $j\in\{0,\ldots,m-1\}$.

2- The associated graded algebra $\mathrm{Gr}(B_{n,e,l,Q})=\oplus_{i\in\mathbb{N}}\overline{B}_{i}$,
where $\overline{B}_{i}=\mathcal{F}_{i}/\mathcal{F}_{i-1}$, is generated
by $\overline{x}=gr_{\partial}(x)$, $\overline{y}=gr_{\partial}(y)$,
$\overline{z}=gr_{\partial}(z)$, $\overline{t}=gr_{\partial}(t)$
as an algebra over $k$ with the relation $\overline{y}(\overline{x}^{n}+\overline{z}^{e})^{l}=\overline{t}^{m}$,
i.e. $\mathrm{Gr}(B_{n,e,l,Q})=k[\overline{X},\overline{Y},\overline{Z},\overline{T}]/\langle\overline{Y}(\overline{X}^{n}+\overline{Z}^{e})^{l}-\overline{T}^{m}\rangle$.
And we have : 
\[
\overline{B}_{mi+j}=k[\overline{x},\overline{z}]\overline{t}^{j}\overline{y}^{i}
\]
where $i\in\mathbb{N}$, and $j\in\{0,\ldots,m-1\}$.
\begin{prop}
With the notation above we have:

\emph{(1)} $\mathrm{ML}(B_{n,e,l,Q})=k[x,z]$. Consequently $B$ is
semi-rigid.

\emph{(2)} Every $D\in\mathrm{LND}(B_{n,e,l,Q})$ has the form $D=f(x,z)\partial$.\end{prop}
\begin{proof}
(1) Given a non-zero $D\in\mathrm{LND}(B)$. By Proposition \ref{Daigle}
$D$ induces a non-zero $\overline{D}\in\mathrm{LND}\left(\mathrm{Gr}(B)\right)$.
Suppose that $f\in\ker(D)\setminus k$, then $\overline{f}\in\ker(\overline{D})\setminus k$.
So there exists $i\in\mathbb{N}$ such that $\overline{f}\in\overline{B}_{i}$.
Assume that $\overline{f}\notin k[\overline{x},\overline{z}]=\overline{B}_{0}$,
then one of the elements $\overline{t}$, $\overline{y}$ must divides
$\overline{f}$ (see \S \ref{sub:The-second-type} above). Which
leads to a contradiction as follows: 

If $\overline{t}$ divides $\overline{f}$ , then $\overline{t}\in\ker(\overline{D})$
as $\ker(\overline{D})$ is factorially closed, and for the same reason
$\overline{y},(\overline{x}^{n}+\overline{z}^{e})^{l}\in\ker(\overline{D})$
due to the relation $\overline{y}(\overline{x}^{n}+\overline{z}^{e})^{l}=\overline{t}^{m}$.
So $\overline{x}^{n}+\overline{z}^{e}\in\ker(\overline{D})$ which
implies that $\overline{x},\overline{z}\in\ker(\overline{D})$ (\cite{key-gene}
Lemma 9.3). This means $\overline{D}=0$, a contradiction.

Finally, if $\overline{y}$ divides $\overline{f}$, then $\overline{D}(\overline{y})=0$.
Choose $H\in\ker(\overline{D})$ which is homogeneous and algebraically
independent of $\overline{y}$, which is possible, since $\mathrm{tr.deg_{k}}\ker(\overline{D})=2$
and $\ker(\overline{D})$ is generated by homogeneous elements. Then
by \S \ref{sub:The-second-type}, $H$ has the form $H=h(\overline{x},\overline{z}).\overline{y}^{l}$.
By algebraic dependence, we may assume $H=h(\overline{x},\overline{z})$,
which is non-constant, and that $h(0,0)=0$. So $\overline{D}$ extends
to a locally nilpotent derivation $\widetilde{D}$ of the ring $\widetilde{B}=k(\overline{y},H)${[}$\overline{x},\overline{z},\overline{t}]/\langle h(\overline{x},\overline{z})-H,\overline{y}(\overline{x}^{n}+\overline{z}^{e})^{l}-\overline{t}^{m}\rangle$.
But $\widetilde{B}$ is of transcendence degree one over the field
$k(\overline{y},H)$ whose spectrum has a singular point at $0$.
This means that $\widetilde{B}$ is rigid (Lemma \ref{rem rigid example}).
Thus $\widetilde{D}=0$, which implies $\overline{D}=0$, a contradiction.

So the only possibility is that $\overline{f}\in k[\overline{x},\overline{z}]$,
and this means $\deg_{\partial}(f)=0$, thus $f\in k[x,z]$ and $\ker(D)\subset k[x,z]$.
Finally, $k[x,z]=\ker(D)$ because $\mathrm{tr.deg_{k}}(\ker(D))=2$.
So we get $\mathrm{ML}(B)=k[x,z]$

(2) follows again from Proposition \ref{Pro:alri-rigid one dimensional}.
\end{proof}

\section{\textbf{A new class of semi-rigid rings}}

In this section, we use the $\mathrm{LND}$-filtration method to establish
the semi-rigidity of new families of two dimensional domains of the
form 
\[
R=k[X,Y,Z]/\left\langle X^{n}Y-P\left(X,Q(X,Y)-X^{e}Z\right)\right\rangle 
\]
for suitable integers $n,e\geq2$ and polynomials $P(X,T),Q(X,T)\in k[X,T]$.
They share with Danielewski hypersurfaces discussed in \ref{ex:Daniel}
above, the property to come naturally equipped with an irreducible
locally nilpotent derivation induced by a locally nilpotent derivation
of $k[X,Y,Z]$. But in contrast with the Danielewski hypersurfaces
case, the corresponding derivation on $k[X,Y,Z]$ are no longer triangular,
in fact not even triangulable by virtue of characterization due to
Daigle \cite{key-Daigle}.

We will begin with a very elementary example illustrating the steps
needed to determine the $\mathrm{LND}$-filtration and its associated
graded algebra, and then we proceed to the general case.

\subsection{\label{dd}A toy example}

\indent\newline\noindent  We let 
\[
R=k[X,Y,Z]/\langle X^{2}Y-(Y^{2}-XZ)^{2}\rangle
\]
 and we let $x$, $y$, $z$ be the images of $X$, $Y$, $Z$ in
$R$. A direct computation reveals that the derivation 
\[
2XS\partial_{Y}+(4YS-X^{2})\partial_{Z}
\]
of $k[X,Y,Z]$ where $S:=Y^{2}-XZ$ is locally nilpotent and annihilates
the polynomial $X^{2}Y-(Y^{2}-XZ)^{2}$. It induces a locally nilpotent
derivation $\partial$ of $R$ for which we have $\partial(x)=0$,
$\partial^{3}(y)=0,\partial^{5}(z)=0$. Furthermore, the element $s=y^{2}-xz$
is a local slice for $\partial$ with $\partial(s)=x^{3}$. The kernel
of $\partial$ is $k[x]$ and the plinth ideal is the principal ideal
generated by $x^{3}$. We have $\deg_{\partial}(x)=0$, $\deg_{\partial}(y)=2$,
$\deg_{\partial}(z)=4$, $\deg_{\partial}(s)=1$.
\begin{prop}
\label{Prop:main-result-filtration} With the notation above, we have:

\emph{(1) }The $\partial$-filtration $\{\mathcal{F}_{i}\}_{i\in\mathbb{N}}$
is given by : 
\[
\mathcal{F}_{4i+2j+l}=k[x]s^{l}y^{j}z^{i}+\mathcal{F}_{4i+2j+l-1}
\]
where $i\in\mathbb{N}$, $j\in\{0,1\}$, $l\in\{0,1\}$.

\emph{(2)} The associated graded algebra $\mathrm{Gr_{\partial}}(R)=\oplus_{i\in\mathbb{N}}\overline{R}_{i}$,
where $\overline{R}_{i}=\mathcal{F}_{i}/\mathcal{F}_{i-1}$, is generated
by $\overline{x}=gr_{\partial}(x)$, $\overline{y}=gr_{\partial}(y)$,
$\overline{z}=gr_{\partial}(z)$, $\overline{s}=gr_{\partial}(s)$
as an algebra over $k$ with relations $\overline{x}^{2}\overline{z}=\overline{s}^{2}$
and $\overline{x}\,\overline{z}=\overline{y}^{2}$, i.e. $\mathrm{Gr_{\partial}}(R)=k[\overline{X},\overline{Y},\overline{Z},\overline{S}]/\langle\overline{X}^{2}\overline{Z}-\overline{S}^{2},\overline{X}\overline{Z}-\overline{Y}^{2}\rangle$.
Furthermore:
\[
\overline{R}_{4i+2j+l}=k[\overline{x}]\overline{s}^{l}\overline{y}^{j}\overline{z}^{i}
\]
where $i\in\mathbb{N}$, $j\in\{0,1\}$, $l\in\{0,1,2,3\}$.\end{prop}
\begin{proof}
1) First, the $\partial$-filtration \textit{$\{\mathcal{F}_{i}\}_{i\in\mathbb{N}}$}
is given by $\mathcal{F}_{r}=\sum_{h\leq r}H_{h}$ where $H_{h}:=\sum_{u+2v+4w=h}k[x]\left(s^{u}y^{v}z^{w}\right)$
and $u,v,w,h\in\mathbb{N}$. To show this, let $J$ be the ideal in
$k^{[4]}=k[X,Y,Z,S]$ defined by $J=\left(X^{2}Y-(Y^{2}-XZ)^{2},Y^{2}-XZ-S\right)$.
Define an $\mathbb{N}$-degree function $\omega$ on $k^{[4]}$ by
declaring that $\omega(X)=0$, $\omega(S)=1$, $\omega(Y)=2$, and
$\omega(Z)=4$. By Lemma \ref{Pro:when-it-is-proper}, the $\mathbb{N}$-filtration
\textit{$\{\mathcal{G}_{r}\}_{i\in\mathbb{N}}$} where $\mathcal{G}_{r}=\sum_{h\leq r}H_{h}$
is proper if and only if $\hat{J}$ is prime. Which is the case since
$\hat{J}=\left\langle X^{2}Y-S^{2},Y^{2}-XZ\right\rangle $ is prime.
Thus by Lemma \ref{lem:equal-if-it-is-proper} we get the desired
description.

Second, let \textit{$l\in\{0,1\}$ }\textit{\emph{and }}\textit{$j\in\{0,1,2,3\}$
}\textit{\emph{be such that }} $l:=r$ mod $2$, $j:=r-l$ mod $4$,
and $i:=\frac{r-2j-l}{4}$. Then we get the following unique expression
$r=4i+2j+l$. Since $\mathcal{F}_{r}=\sum_{u+2v+4w=r}k[x]\left(s^{u}y^{v}z^{w}\right)+\mathcal{F}_{r-1}$,
we conclude in particular that $\mathcal{F}_{r}\supseteq k[x]s^{l}y^{j}z^{i}+\mathcal{F}_{r-1}$.
For the other inclusion, the relation $x^{2}y=s^{2}$ allows to write
$s^{u}y^{v}z^{w}=x^{e}s^{l}y^{v_{0}}z^{w}$ and from the relation
$y^{2}=s+zx$ we get $x^{e}s^{l}y^{v_{0}}z^{w}=x^{e}s^{l}y^{j}(s+xz)^{n}z^{w}$.
Since the monomial with the highest degree relative to $\deg_{\partial}$
in $(s+xz)^{n}$ is $x^{n}.z^{n}$, we deduce that $x^{e}s^{l}y^{j}(s+xz)^{n}z^{w}=x^{e+n}s^{l}y^{j}z^{w+n}+\sum M_{\beta}$
where $M_{\beta}$ is monomial in $x$, $y$, $s$, $z$ of degree
less than $r$. Since the expression $r=4i+2j+l$ is unique, we get
$w+n=i$. So $s^{u}y^{v}z^{w}=x^{e+n}s^{l}y^{j}z^{i}+f$ where $f\in\mathcal{F}_{r-1}$.
Thus $k[x]\left(s^{u}y^{v}z^{w}\right)\subseteq k[x]s^{l}y^{j}z^{i}+\mathcal{F}_{r-1}$
and finally $\mathcal{F}_{r}=k[x]s^{l}y^{j}z^{i}+\mathcal{F}_{r-1}$.

2) By part (1), an element $f$ of degree $r$ can be written as $f=g(x)s^{l}y^{j}z^{i}+f_{0}$
where $f_{0}\in\mathcal{F}_{r-1}$, $l=r$ mod $2$, $j=r-l$ mod
$4$, $i=\frac{r-2j-l}{4}$, and $i\in\mathbb{N}$, $j\in\{0,1\}$,
$l\in\{0,1\}$. So by Lemma \ref{lem:graded relation}, P2, P1, P3
respectively we get 
\[
\overline{f}=\overline{g(x)s^{l}y^{j}z^{i}+h}=\overline{g(x)s^{l}y^{j}z^{i}}=\overline{g(x)}\overline{s}^{l}\overline{y}^{j}\overline{z}^{i}=g(\overline{x})\overline{s}^{l}\overline{y}^{j}\overline{z}^{i}
\]
 and therefore $\overline{B}_{4i+2j+l}=k[\overline{x}]\overline{s}^{l}\overline{y}^{j}\overline{z}^{i}$. 

Finally, by Proposition \ref{prop:graded algebra}, $\mathrm{Gr}_{\partial}(B)=k[\overline{X},\overline{Y},\overline{Z},\overline{S}]/\langle\overline{X}^{2}\overline{Z}-\overline{S}^{2},\overline{X}\,\overline{Z}-\overline{Y}^{2}\rangle$.
\end{proof}

\subsection{\label{a more general case}A more general family}

\indent\newline\noindent  We now consider more generally rings $R$
of the form 
\[
k[X,Y,Z]/\left\langle X^{n}Y-P\left(X,Q(X,Y)-X^{e}Z\right)\right\rangle 
\]
 where 
\[
P(X,S)=S^{d}+f_{d-1}(X)S^{d-1}+\cdots+f_{1}(X)S+f_{0}(X)
\]
\[
Q(X,Y)=Y^{m}+g_{m-1}(X)Y^{m-1}+\cdots+g_{1}(X)Y+g_{0}(X)
\]
 $n\geq2$, $d\geq2$, $m\geq1$, and $e\geq1$. Up to a change of
variable of the form $Y\mapsto Y-c$ where $c\in k$, we may assume
that $0\in\mathrm{Spec}(R)$. 

Let $x$, $y$, $z$ be the images of $X$, $Y$, $Z$ in $R$. Define
$\partial$ by: $\partial(x)=0$, $\partial(s)=x^{n+e}$ where $s:=Q(x,y)-x^{e}z$.
Considering the relation $x^{n}y=P(x,Q(x,y)-x^{e}z)$, a simple computation
leads to $\partial(y)=x^{e}\frac{\partial P}{\partial s}$ ,$\partial(z)=\frac{\partial Q}{\partial y}\frac{\partial P}{\partial s}-x^{n}$,
i.e. 
\[
\partial:=x^{e}\frac{\partial P}{\partial s}\partial_{y}+(\frac{\partial Q}{\partial y}\,\frac{\partial P}{\partial s}-x^{n})\partial_{z}
\]
where $\frac{\partial P}{\partial s}=ds^{d-1}+(d-1)f_{d-1}(x)s^{d-2}+\cdots+f_{1}(x)$,
and $\frac{\partial Q}{\partial y}=my^{m-1}+(m-1)g_{m-1}(x)y^{m-2}+\cdots+g_{1}(x)$.
Since $\partial(x^{n}y-P(x,Q(x,y)-x^{e}z))=0$ and $\partial^{d+1}(y)=0,\partial^{md+1}(z)=0$,
$\partial$ is a well-defined locally nilpotent derivation of \textbf{$R$}.
The kernel of $\partial$ is equal to $k[x]$ and the element $s$
is a local slice for $\partial$ by construction. One checks further
that the plinth ideal is equal to $\mathrm{pl}(\partial)=\left\langle x^{n+e}\right\rangle $.
A direct computation shows that $\deg_{\partial}(x)=0$, $\deg_{\partial}(y)=d$,
$\deg_{\partial}(z)=md$ and $\deg_{\partial}(s)=1$. Furthermore:

1- The $\partial$-filtration $\{\mathcal{F}_{i}\}_{i\in\mathbb{N}}$
is given by : 
\[
\mathcal{F}_{mdi+dj+l}=k[x]s^{l}y^{j}z^{i}+\mathcal{F}_{mdi+dj+l-1}
\]
where $i\in\mathbb{N}$, $j\in\{0,\ldots,m-1\}$, $l\in\{0,\ldots,d-1\}$.

2- The associated graded algebra $\mathrm{Gr}(R)=\oplus_{i\in\mathbb{N}}\overline{R}_{i}$,
where $\overline{R}_{i}=\mathcal{F}_{i}/\mathcal{F}_{i-1}$, is generated
by $\overline{x}=gr_{\partial}(x)$, $\overline{y}=gr_{\partial}(y)$,
$\overline{z}=gr_{\partial}(z)$, $\overline{s}=gr_{\partial}(s)$
as an algebra over $k$ with relations $\overline{x}^{n}\overline{z}=\overline{s}^{d}$
and $\overline{x}^{e}\,\overline{z}=\overline{y}^{m}$ , i.e. $\mathrm{Gr}(R)=k[\overline{X},\overline{Y},\overline{Z},\overline{S}]/\left\langle \overline{X}^{n}\overline{Z}-\overline{S}^{d},\,\overline{X}^{e}\,\overline{Z}-\overline{Y}^{m}\right\rangle $.
And we have : 
\[
\overline{R}_{mdi+dj+l}=k[\overline{x}]\overline{s}^{l}\overline{y}^{j}\overline{z}^{i}
\]
where $i\in\mathbb{N}$, $j\in\{0,\ldots,m-1\}$, $l\in\{0,\ldots,d-1\}$.
\begin{thm}
\label{fdsa} With the above notation the following hold:

\emph{(1) }$\mathrm{ML}(R)=k[x]$. Consequently $R$ is semi-rigid.

\emph{(2)} Every $D\in\mathrm{LND}(R)$ has form $D=f(x)\partial$,
i.e. $R$ is almost rigid.\end{thm}
\begin{proof}
(1) Given a non-zero $D\in\mathrm{LND}(R)$. By Proposition \ref{Daigle},
$D$ respects the $\partial$-filtration and induces a non-zero locally
nilpotent derivation $\overline{D}$ of $\mathrm{Gr}(R)$. Suppose
that $f\in\ker(D)\setminus k$, then $\overline{f}\in\ker(\overline{D})\setminus k$
is an homogenous element of $\mathrm{Gr}(R)$. So there exists $i\in\mathbb{N}$
such that $\overline{f}\in\overline{R}_{i}$.

Assume that $\overline{f}\notin k[\overline{x}]=\overline{R}_{0}$,
then one of the elements $\overline{s}$, $\overline{y}$, $\overline{z}$
must divides $\overline{f}$ by \ref{a more general case},2. Which
leads to a contradiction as follows : 

If $\overline{s}$ divides $\overline{f}$ , then $\overline{s}\in\ker(\overline{D})$
as $\ker(\overline{D})$ is factorially closed, and for the same reason
$\overline{x},\overline{y}\in\ker(\overline{D})$ due to the relation
$\overline{x}^{n}\overline{y}=\overline{s}^{d}$. Then by the relation
$\overline{x}^{e}\,\overline{z}=\overline{y}^{m}$, we must have $\overline{z}\in\ker(\overline{D})$,
which means $\overline{D}=0$, a contradiction. In the same way, we
get a contradiction if $\overline{y}$ divides $\overline{f}$. 

Finally, if $\overline{z}$ divides $\overline{f}$, then $\overline{D}(\overline{z})=0$.
So $\overline{D}$ induces in a natural way a locally nilpotent derivation
$\widetilde{D}$ of the ring $\widetilde{R}=k(\overline{z})[\overline{x},\overline{y},\overline{s}]/\langle\overline{x}^{n}\,\overline{z}-\overline{s}^{d},\overline{x}^{e}\,\overline{z}-\overline{y}^{m}\rangle$.
But since $0\in\mathrm{Spec(\widetilde{R})}$ is a singular point,
$\widetilde{R}$ is rigid (Lemma \ref{rem rigid example}). So $\widetilde{D}=0$,
which implies $\overline{D}=0$, a contradiction.

So the only possibility is that $\overline{f}\in k[\overline{x}]$,
and this means $\deg_{\partial}(f)=0$, thus $f\in k[x]$ and $\ker(D)\subset k[x]$.
Finally, $k[x]=\ker(D)$ because $\mathrm{tr.deg_{k}}(\ker(D))=1$
and $k[x]$ is algebraically closed in $B$. So $\mathrm{ML}(R)=k[x]$.

(2) follows again from Proposition \ref{Pro:alri-rigid one dimensional}.
\end{proof}

\section{\textbf{Further applications of the $\mathrm{LND}$-filtratoin}}

Given a commutative domain $B$ over an algebraically closed field
$k$ of characteristic zero, we denote $\mathrm{Aut_{k}}(B)$ the
group of algebraic $k$-automorphisms of $B$. This group acts by
conjugation on $\mathrm{LND}(B)$. An immediate consequence is that
$\alpha(\mathrm{ML}(B))=\mathrm{ML}(B)$ for every $\alpha\in\mathrm{Aut_{k}}(B)$
which yield in particular an induced action of $\mathrm{Aut_{k}}(B)$
on $\mathrm{ML}(B)$. Let $\partial_{\alpha}=\alpha^{-1}\partial\alpha$
be the conjugate of $\partial$ by a given automorphism $\alpha$
of $B$, it is straightforward to check that $\alpha\{\ker(\partial_{\alpha})\}=\ker(\partial)$
and more generally that $\deg_{\partial_{\alpha}}(b)=\deg_{\partial}(\alpha(b))$
for any $b\in B$. In other words, $\alpha$ respects $\deg_{\partial}$
and $\deg_{\partial_{\alpha}}$(i.e. $\alpha$ sends an element of
degree $n$ relative to $\deg_{\partial_{\alpha}}$, to an element
of the same degree $n$ relative to $\deg_{\partial}$. 
\begin{defn}
We say that an algebraic $k$-automorphism $\alpha$ preserves the
\textit{$\partial$-filtration} for some $\partial\in\mathrm{LND}(B)$
if $\deg_{\partial}(\alpha(b))=\deg_{\partial}(b)$ for any $b\in B$.\end{defn}
\begin{lem}
\label{Lem:preserve-filtration} Let $\partial\in\mathrm{LND}(B)$
and $\alpha\in\mathrm{Aut_{k}}(B)$. Then $\partial$ and $\partial_{\alpha}$
are equivalent, i.e. have the same kernel, if and only if $\alpha$
preserve the $\partial$-filtration. \end{lem}
\begin{proof}
Suppose that \textit{$\partial$ }\textit{\emph{and}}\textit{ $\partial_{\alpha}$}
are equivalents, then $\deg_{\partial}(\alpha(b))=\deg_{\partial_{\alpha}}(b)$
for every $b\in B$, and by hypothesis $\ker(\partial)=\ker(\partial_{\alpha})$,
so $\deg_{\partial}=\deg_{\partial_{\alpha}}$. Then we obtain $\deg_{\partial}(\alpha(b))=\deg_{\partial}(b)$.
Thus $\alpha$ preserves the $\partial$-filtration. Since the other
direction is obvious we are done.
\end{proof}
The following Corollary shows a nice property of a semi-rigid ring.
That is, every algebraic automorphism $\alpha$ has to preserve the
unique filtration induced by any locally nilpotent derivation $\partial$,
i.e. $\alpha$ sends an element of degree $i$ relative to $\partial$
to an element of the same degree relative to $\partial$. Which makes
the computation of the group of automorphisms easier up to the automorphism
group of $\ker(\partial)$. 
\begin{cor}
\label{Cor:preserve-almost-rigid-filtration}Let $B$ be a semi-rigid
$k$-domain, then every $k$-automorphism of $B$ preserve the $\partial$-filtration
for every $\partial\in\mathrm{LND}(B)$.\end{cor}
\begin{proof}
A direct consequence of Definition \ref{def almost-rigid} and Lemma
\ref{Lem:preserve-filtration}.
\end{proof}

\subsection{The group of algebraic $k$-automorphisms of a semi-rigid $k$-domain}

\indent\newline\noindent  Suppose that $B$ is a semi-rigid $k$-domain.
Then, it has a unique proper filtration $\{\mathcal{F}_{i}\}_{i\in\mathbb{N}}$
which is the $\partial$-filtration corresponding to any non-zero
locally nilpotent derivation $\partial$ of $B$. Since every algebraic
$k$-automorphism of $B$ preserves this filtration (Corollary \ref{Cor:preserve-almost-rigid-filtration}),
we obtain an exact sequence 
\[
0\rightarrow\mathrm{Aut_{k}}\left(B,\mathrm{ML}(B)\right)\rightarrow\mathrm{Aut_{k}}(B)\rightarrow\mathrm{Aut_{k}}\left(\mathrm{ML}(B)\right)
\]
where $\mathrm{Aut_{k}}\left(B,\mathrm{ML}(B)\right)$ is by definition
the sub-group of $\mathrm{Aut_{k}}(B)$ consisting of elements whose
induced action on $\mathrm{ML}(B)$ is trivial. Furthermore, every
element of $\mathrm{Aut_{k}}\left(B,\mathrm{ML}(B)\right)$ induces
for every $i\geq1$ an automorphism of $\mathcal{F}_{0}$-module of
each $\mathcal{F}_{i}$. In this section we illustrate how to exploit
these information to compute $\mathrm{Aut_{k}}(B)$ for certain semi-rigid
$k$-domains $B$.

\subsubsection{\label{sub}\emph{ }\textbf{\emph{$\mathrm{Aut_{k}}$ for example}}\emph{
}\ref{ex:Daniel}}

\indent\newline\noindent  In \cite{key-Makar} Makar-Limanov computed
the $k$-automorphism group for surfaces in $k^{[3]}$ defined by
equation $X^{n}Z-P(Y)=0$ where $n>1$ and $\deg_{Y}P(Y)>1$. Then
Poloni, see \cite{key-Poloni}, generalized Makar-Limanov's method
to obtained similar results for the rings considered in \ref{ex:Daniel}
above. Here we briefly indicate how to recover these results using
$\mathrm{LND}$-filtrations. So let 
\[
B_{n,P}=k[X,Y,Z]/\left\langle X^{n}Z-P(X,Y)\right\rangle 
\]
where \textbf{$P(X,Y)=Y^{m}+f_{m-1}(X)Y^{m-1}+\cdots+f_{0}(X)$, $f_{i}(X)\in k[X]$}{\scriptsize ,}
$n\geq2$, and $m\geq2$. Up to change of variable of the form $Y$
by $Y-\frac{f_{m-1}(X)}{m}$ we may assume without loss of generality
that $f_{m-1}(X)=0$. 
\begin{prop}
\label{nnn}Let $B_{n,P}$ be as above. Then every algebraic $k$-automorphism
$\alpha$ of $B$ has the form 
\[
\alpha(x,y,z)=(\lambda x,\,\mu y+x^{n}a(x),\,\frac{\mu^{m}}{\lambda^{n}}z+\frac{P(\lambda x,\mu y+x^{n}a(x))-\mu^{m}P(x,y)}{\lambda^{n}x^{n}})
\]
where $\lambda,\mu\in k^{*}$ satisfy $f_{m-i}(\lambda x)\equiv\mu^{i}.f_{m-i}(x)$
mod $x^{n}$ for all $i$, and $a(x)\in k[x]$.\end{prop}
\begin{proof}
By Proposition \ref{prop:Daniel-example}, (1) and \S \ref{ex:Daniel},
$\mathrm{ML}(B)=k[x]$ and the $\partial$-filtration $\{\mathcal{F}_{i}\}_{i\in\mathbb{N}}$
is given by $\mathcal{F}_{im+j}=k[x]y^{im+j}+k[x]y^{im+j-m}z+\ldots+k[x]y^{j}z^{i}+\mathcal{F}_{im+j-1}$,
where $\partial=x^{n}.\partial_{y}+\frac{\partial P}{\partial y}.\partial_{z}$,
$\deg_{\partial}(x)=0$, $\deg_{\partial}(y)=1$, and $\deg_{\partial}(z)=m$.
In particular $\mathcal{F}_{0}=k[x]$, $\mathcal{F}_{1}=k[x]y+\mathcal{F}_{0}$,
and $\mathcal{F}_{m}=k[x]y^{m}+k[x]z+\mathcal{F}_{m-1}$. 

Now by Corollary \ref{Cor:preserve-almost-rigid-filtration} $\alpha$
preserve $\deg_{\partial}$, so we must have $\alpha(x)\in\mathcal{F}_{0}=k[x]$,
$\alpha(y)\in\mathcal{F}_{1}=k[x]y+k[x]$ and $\alpha(z)\in\mathcal{F}_{m}=k[x]y^{m}+k[x].z+\mathcal{F}_{m-1}$.
Since $\alpha$ is invertible we get $\alpha(x)=\lambda x+c$, $\alpha(y)=\mu y+b(x)$,
and $\alpha(z)=\xi z+h(x,y)$ where $\lambda,\mu,\xi\in k^{*}$, $c\in k$,
$b\in k[x]$, $h(x,y)\in k[x,y]$, and $\deg_{y}h(x,y)\leq m$.

By Proposition \ref{prop:Daniel-example} (2) every $D\in\mathrm{LND}(B)$
has the form $D=f(x)\partial$. In particular, $\partial_{\alpha}=f(x)\partial$
for some $f(x)\in k[x]$. Since $\alpha\partial_{\alpha}=\partial\alpha$
we have $\partial(\alpha(y))=\alpha(f(x)\partial(y))=f(\alpha(x))\alpha(x^{n})$
where ($\partial(y)=x^{n}$). So we get $\partial(\mu y+b(x))=f(\alpha(x))\left(\lambda x+c\right)^{n}$.
Since $\partial(\mu y+b(x))=\mu x^{n}$, $x$ divides $(\lambda x+c)^{n}$
in $k[x]$, and this is possible only if $c=0$, so we get $\alpha(x)=\lambda x$.

Applying $\alpha$ to the relation $x^{n}z=P(x,y)$ in $B_{P,n}$,
we get $\lambda^{n}x^{n}\alpha(z)=P(\lambda x,\mu y+b(x))=\mu^{m}P(x,y)+m\mu^{m-1}y^{m-1}b(x)+H(x,y)$
where $\deg_{y}H\leq m-2$. Since $x^{n}$ divides $P=x^{n}z$ and
$\deg_{y}H\leq m-2$, $x^{n}$ divides $m\mu^{m-1}y^{m-1}b(x)+H(x,y)$
in $k[x,y]$. So $x^{n}$ divides $b(x)$, i.e. $\alpha(y)=\mu y+x^{n}a(x)$. 

In addition, $x^{n}$ divides every coefficient of $H$ as a polynomial
in $y$, so $x^{n}$ divides $-\mu^{m}f_{m-i}(x)+\mu^{m-i}f_{m-i}(\lambda x)$
because coefficients of $H(y)$ are of the form $q(x,y)b(x)-\mu^{m}f_{m-i}(x)+\mu^{m-i}f_{m-i}(\lambda x)$
and $b(x)$ is divisible by $x^{n}$. So $x^{n}$ divides $-\mu^{i}f_{m-i}(x)+f_{m-i}(\lambda x)$
for every $i$. And we are done. 
\end{proof}

\subsubsection{\textbf{\emph{$\mathrm{Aut_{k}}$ for example}}\textbf{ }\ref{a more general case}}

\indent\newline\noindent  The same method as in \ref{sub} can be
applied to compute automorphism groups of rings $R$ defined as in
Theorem \ref{fdsa}. For simplicity we only deal with the case where
$Q(X,Y)=Y^{m}$, the general case can be deduced in the same way at
the cost of longer and more complicated computation. Again we make
a substitution in $S$ as in \ref{sub} to get relation of the form
presented in the following result:
\begin{thm}
\label{ll} Let $R$ denote the ring $R=k[X,Y,Z]/\left\langle X^{n}Y-P(X,Y^{m}-X^{e}Z)\right\rangle =k[x,y,z]$
where $P(X,S)=S^{d}+f_{d-2}(X)S^{d-2}+\cdots+f_{1}(X)S+f_{0}(S)$,
$n\geq2$, $d\geq2$, $m\geq1$, and $e\in\mathbb{N}$. Then, every
algebraic $k$-automorphism of $B$ has the form 
\[
\alpha(x,y,z)=(\lambda x,\,\frac{\mu^{d}}{\lambda^{n}}y+F,\,\frac{\mu^{dm}}{\lambda^{nm+e}}.z+\frac{(\frac{\mu^{d}}{\lambda^{n}}y+F)^{m}-\frac{\mu^{dm}}{\lambda^{nm}}y^{m}+x^{n+e}a(x)}{\lambda x^{e}})
\]
where:

$\lambda,\mu\in k^{*}$ verify both $\frac{\mu^{dm}}{\lambda^{nm}}=\mu$
and $f_{d-i}(\lambda x)\equiv\mu^{i}f_{d-i}(x)$ mod $x^{n+e}$ for
every $i\in\{2,\ldots,d\}$. $s=y^{m}-x^{e}z$ , and $F=\frac{P(\lambda x,\mu s+x^{n+e}a(x))-\mu^{d}P(x,s)}{\lambda^{n}x^{n}}$,
$a(x)\in k[x]$.\\
\end{thm}
\begin{proof}
A similar argument as in the proof of Proposition \ref{nnn} leads
to $\alpha(x)=\lambda x$ and $\alpha(s)=\mu s+x^{n}a(x)$ where $\lambda,\mu\in k^{*}$
verify $f_{d-i}(\lambda x)\equiv\mu^{i}f_{d-i}(x)$ mod$x^{n}$ for
all $i$. Now $\alpha(x)$ and $\alpha(s)$ determine 
\[
\alpha(y)=\frac{\mu^{d}}{\lambda^{n}}y+\frac{P(\lambda x,\mu s+x^{n}a(x))-\mu^{d}P(x,s)}{\lambda^{n}x^{n}}.
\]

Apply $\alpha$ to $x^{e}z=y^{m}-s$ to get $\lambda^{e}x^{e}\alpha(z)=(\frac{\mu^{h}}{\lambda^{n}}y+F)^{m}-\mu s-x^{n}a(x)$
where $F=\frac{P(\lambda x,\mu s+x^{n}a(x))-\mu^{h}P(x,s)}{\lambda^{n}x^{n}}\in k[x,s]$.
So we have $\lambda^{e}x^{e}\alpha(z)=[\frac{\mu^{hm}}{\lambda^{nm}}y^{m}-\frac{\mu^{hm}}{\lambda^{nm}}s]+(\frac{\mu^{hm}}{\lambda^{nm}}-\mu)s+m(\frac{\mu^{h}}{\lambda^{n}}y)^{m-1}F+\ldots+F^{m}-x^{n}a(x)$.
Since $\frac{\mu^{hm}}{\lambda^{nm}}y^{m}-\frac{\mu^{hm}}{\lambda^{nm}}s=\frac{\mu^{hm}}{\lambda^{nm}}x^{e}z$,
we see that $x^{e}$ divides $G:=(\frac{\mu^{hm}}{\lambda^{nm}}-\mu)s+m(\frac{\mu^{h}}{\lambda^{n}}y)^{m-1}F+\ldots+F^{m}-x^{n}a(x)$
in $k[x,s,y]\subset B$ because $F\in k[x,s]$. Thus $x^{e}$ divides
every coefficients of $G$ as a polynomial in $y$ because $\deg_{s}G\leq d-1$.
So $x^{e}$ divides $F$ and $\frac{\mu^{hm}}{\lambda^{nm}}-\mu=0$.
This means that $x^{n+e}$ divides $f_{d-i}(\lambda x)-\mu^{i}f_{d-i}(x)$
for all $i\in\{2,\ldots,d\}$ (see proof of Proposition \ref{nnn}),
and $\frac{\mu^{hm}}{\lambda^{nm}}=\mu$. Finally, by the relation
\textit{$s=y^{m}-x^{e}z$,} we get $\alpha(z)=\frac{\mu^{hm}}{\lambda^{nm+e}}z+\frac{m(\frac{\mu^{h}}{\lambda^{n}}y)^{m-1}F+\ldots+F^{m}+x^{n}a(x)}{\lambda x^{e}}$,
and we are done.
\end{proof}

\subsection{\label{ISO} Isomorphisms}

\indent\newline\noindent  We are going to use the previous facts
about semi-rigid $k$-domains to give a necessary and sufficient condition
for two hypersurfaces of the family \ref{ex:Daniel} to be isomorphic. 

Let $\Psi:A\longrightarrow B$ be an algebraic isomorphism, we refer
to this by $\Psi\in\mathrm{Isom}{}_{k}(A,B)$, between two finitely
generated $k$-domains $A=k[y_{1},\ldots,y_{r}]$, $B=k[x_{1},\ldots,x_{r}]$
where $y_{1},\ldots,y_{r}$ and $x_{1},\ldots,x_{r}$ are minimal
sets of generators. Since $\Psi(A)=k[\Psi(y_{1}),\ldots,\Psi(y_{r})]=k[x_{1},\ldots,x_{r}]$,
there exists an automorphism $\psi:B\longrightarrow B$ such that
for every $i\in\{1,\ldots,r\}$ there exists $j\in\{1,\ldots,r\}$
such that $\psi(x_{i})=\Psi(y_{j})$. 

Given $\partial\in\mathrm{LND}(B)$ and $\Psi\in\mathrm{Isom}{}_{k}(A,B)$.
For any $n\in\mathbb{N}$ we have $(\Psi^{-1}\partial\Psi)^{n}=\Psi^{-1}\partial^{n}\Psi$.
So we see that $\Psi^{-1}\partial\Psi\in\mathrm{LND}(A)$. An immediate
result is $\Psi(\mathrm{ML}(A))=\mathrm{ML}(B)$ for any $\Psi\in\mathrm{Isom}{}_{k}(A,B)$.
Denote $\partial_{\Psi}:=\Psi^{-1}\partial\Psi$,  we have the following
properties 
\begin{enumerate}
\item $\Psi\{\ker(\partial_{\Psi})\}=\ker(\partial)$.
\item $\deg_{\partial_{\Psi}}(a)=\deg_{\partial}(\Psi(a))$ for all $a\in A$.
\end{enumerate}
In \cite{key-Poloni}, Poloni obtained similar results to the next
Proposition.
\begin{prop}
\label{prop:iso}Let $B_{1}=k[x_{1},y_{1},z_{1}]$ and $B_{2}=k[x_{2},y_{2},z_{2}]$
be as in \ref{ex:Daniel} where 
\[
P_{1}=y_{1}^{m_{1}}+f_{m_{1}-2}(x_{1})y_{1}^{m_{1}-2}+\cdots+f_{0}(x_{1})
\]
 
\[
P_{2}=y_{2}^{m_{2}}+g_{m_{2}-2}(x_{2})y_{2}^{m_{2}-2}+\cdots+g_{0}(x_{2})
\]
 $f_{i}(x_{1})\in k[x_{1}]$, $g_{i}(x_{2})\in k[x_{2}]$ and $n_{i}>1,m_{i}>1$.
Then 

$B_{1}$ and $B_{2}$ are isomorphic if and only if $n=n_{1}=n_{2}$,
$m=m_{1}=m_{2}$, and $f_{m-i}(\lambda x)\equiv\mu^{i}g_{m-i}(x)$
mod $x^{n}$ for every $i\in\{2,\ldots,m\}$ .
\end{prop}
\noindent  \textit{In addition, every isomorphism between $B_{1}$
and $B_{2}$ takes the form }\textit{\small 
\[
\Psi(x_{1},y_{1},z_{1})=(\lambda x_{2},\,\mu y_{2}+x_{2}^{n}a(x_{2}),\,\frac{\mu^{m}}{\lambda^{n}}z_{2}+\frac{P_{1}(\lambda x_{2},\mu y_{2}+x_{2}^{n}a(x_{2}))-\mu^{m}P_{2}(x_{2},y_{2})}{\lambda^{n}x_{2}^{n}}),
\]
}\textit{where $a(x)\in k[x]$ and $\lambda\in k^{*}$, $\mu\in k^{*}$
satisfy $f_{m-i}(\lambda x)\equiv\mu^{i}g_{m-i}(x)$ mod $x^{n}$
for all $i$.}
\begin{proof}
Let $D\in\mathrm{LND}(B_{2})$, and let $\Psi\in\mathrm{Isom}{}_{k}(B_{1},B_{2})$.
By property 2, $\deg_{D_{\Psi}}(x_{1})=\deg_{D}(\Psi(x_{1}))$, but
we saw before that $\deg_{E}(x_{1})=0$ for all $E\in\mathrm{LND}(B_{1})$,
so we get $\deg_{D_{\Psi}}(x_{1})=0=\deg_{D}(\Psi(x_{1}))$ and $\Psi(x_{1})\in\mathcal{F}_{0}=k[x_{2}]$.
The same argument shows that $\deg_{D_{\Psi}}(y_{1})=1=\deg_{D}(\Psi(y_{1}))$,
$\Psi(y_{1})\in\mathcal{F}_{1}-\mathcal{F}_{0}$, and $\deg_{D_{\Psi}}(z_{1})=m_{1}=\deg_{D}(\Psi(z_{1}))$.
This implies that the only possibility for $\psi$ defined as in \ref{ISO}
is $\psi(x_{2})=\Psi(x_{1})$, $\psi(y_{2})=\Psi(y_{1})$ and $\psi(z_{2})=\Psi(z_{1})$. 

Now by Proposition \ref{nnn} {\small 
\[
\psi(x_{2},y_{2},z_{2})=(\lambda x_{2},\,\mu y_{2}+x_{2}^{n_{2}}a(x_{2}),\,\frac{\mu^{m_{2}}}{\lambda^{n_{2}}}z_{2}+\frac{Q_{2}(\lambda x_{2},\mu y_{2}+x_{2}^{n_{2}}a(x_{2}))-\mu^{m_{2}}Q_{2}(x_{2},y_{2})}{\lambda^{n_{2}}x_{2}^{n_{2}}}),
\]
}{\small \par}

where $a(x_{2})\in k[x_{2}]$ and $\lambda\in k^{*}$, $\mu\in k^{*}$
such that $g_{m-i}(\lambda x)\equiv\mu^{i}g_{m-i}(x)$ mod $x^{n_{2}}$
for all $i$. So we get

{\small 
\[
\Psi(x_{1},y_{1},z_{1})=(\lambda x_{2},\,\mu y_{2}+x_{2}^{n_{2}}a(x_{2}),\,\frac{\mu^{m_{2}}}{\lambda^{n_{2}}}z_{2}+\frac{Q_{2}(\lambda x_{2},\mu y_{2}+x_{2}^{n_{2}}a(x_{2}))-\mu^{m_{2}}Q_{2}(x_{2},y_{2})}{\lambda^{n_{2}}x_{2}^{n_{2}}}).
\]
}{\small \par}

Since $\deg_{D}(z_{2})=m_{2}$ for any non-zero $D\in\mathrm{LND}(B_{2})$,
and since $\psi$ preserves $\deg_{D}$, we get $m_{2}=\deg_{D}(z_{2})=\deg_{D}(\psi(z_{2}))=\deg_{D}(\Psi(z_{1}))=\deg_{D_{\Psi}}(z_{1})=m_{1}$,
i.e. $m_{1}=m_{2}$.

By applying $\Psi$ to relation $x_{1}^{n_{1}}z_{1}=Q_{1}(x_{1},y_{1})$
in $B_{1}$, we obtain 

{\small 
\[
\lambda^{n_{1}}x_{2}^{n_{1}}(\frac{\mu^{m_{2}}}{\lambda^{n_{2}}}z_{2}+\frac{Q_{2}(\lambda x_{2},\mu y_{2}+x_{2}^{n_{2}}a(x_{2}))-\mu^{m_{2}}Q_{2}(x_{2},y_{2})}{\lambda^{n_{2}}x_{2}^{n_{2}}})=Q_{1}(\lambda x_{2},\mu y_{2}+x_{2}^{n_{2}}a(x_{2})).
\]
}{\small \par}

Applying the map $gr_{D}$ to the last equation, we get

\[
\lambda^{n_{1}}\frac{\mu^{m_{2}}}{\lambda^{n_{2}}}\overline{x}_{2}^{n_{1}}\overline{z}_{2}=\mu^{m_{2}}\overline{y_{2}}^{m_{2}}.
\]

On the other hand $\overline{x}_{2}^{n_{2}}\overline{z}_{2}=\overline{y}_{2}^{m_{2}}$
(apply $gr_{D}$ to $x_{2}^{n_{2}}z_{2}=Q_{2}(x_{2},y_{2})$), the
last two equations give $\lambda^{n_{1}}\frac{\mu^{m_{2}}}{\lambda^{n_{2}}}\overline{x}_{2}^{n_{1}}=\mu^{m_{2}}\overline{x}_{2}^{n_{2}}$
which means that $n_{1}=n_{2}$. We could have obtained that $n_{1}=n_{2}$
in this way: from $\lambda^{n_{1}}x_{2}^{n_{1}}\Psi(z_{1})=Q_{1}(\lambda x_{2},\mu y_{2}+x_{2}^{n_{2}}a(x_{2}))$
we conclude that $\lambda^{n_{1}}x_{2}^{n_{1}}\partial(\Psi(z_{1}))=\partial(y_{2})H(x_{2},y_{2})$
where $\deg_{y_{2}}H<m_{2}$ and $x_{2}$ does not divide $H$. So
$x_{2}^{n_{1}}$ divides $\partial(y_{2})=x_{2}^{n_{2}}$ where $\partial$
is defined as in \ref{ex:Daniel}, which mean that $n_{1}\leq n_{2}$.
Since $B_{1}$ and $B_{2}$ play symmetric roles the equality follows.
\end{proof}
\indent\newline


\noindent  
\begin{thebibliography}{10}
{\normalsize \bibitem{key-gene} G. Freudenburg, }\emph{\normalsize Algebraic
Theory of Locally Nilpotent Derivation}{\normalsize , Encycl. Math.
Sci., 136, Inv. Theory and Alg. Tr. Groups, VII, Springer- Verlag,
2006. }{\normalsize \par}

\bibitem{key-M-L1996}{\normalsize{} L. Makar-Limanov,}\emph{\normalsize{}
On the hypersuface $x+x^{2}y+z^{2}+t^{3}=0$ in $\mathbb{C}^{4}$
or a $\mathbb{C}^{3}$-like threefold which is not $\mathbb{C}^{3}$}{\normalsize ,
Israel J. Math., 96 (1996), pp. 419-429.}{\normalsize \par}

{\normalsize \bibitem{key-ML2001} L. Makar-Limanov, }\emph{\normalsize On
the group of automorphisms of a surface$x^{n}y=P(z)$}{\normalsize ,
Israel J. Math., 121 (2001), pp. 113\textendash{}123.}{\normalsize \par}

{\normalsize \bibitem{key-Kaliman Makar} Sh. Kaliman and L. Makar-Limanov,
}\emph{\normalsize AK-invariant of affine domains}{\normalsize , Affine
Algebraic Geometry}\emph{\normalsize , }{\normalsize 231\textendash{}255,
Osaka University Press, 2007.}{\normalsize \par}

{\normalsize \bibitem{key-sh and makar} Sh. Kaliman and L. Makar-Limanov,
}\emph{\normalsize On the Russell-Koras contractible threefolds}{\normalsize ,
J. Algebraic Geom., 6 no. 2 (1997), 247\textendash{}268}{\normalsize \par}

{\normalsize \bibitem{key-David Maubach} D. Finston and S. Maubach,
}\emph{\normalsize Constructing (almost) rigid rings and a UFD having
infinitely generated Derksen and Makar-Limanov invariant}{\normalsize ,
Canad. Math. Bull.}{\normalsize \par}

{\normalsize \bibitem{key-Makar} L. Makar-Limanov, }\emph{\normalsize Locally
nilpotent derivations, a new ring invariant and applications}{\normalsize ,
available at http://www.math.wayne.edu/\textasciitilde{}lml/lmlnotes.pdf.}{\normalsize \par}

{\normalsize \bibitem{key-Crachiola Makar} A. Crachiola and L. Makar-Limanov,
}\emph{\normalsize An algebraic proof of a cancellation theorem for
surfaces}{\normalsize , J. Algebra 320 (2008), no. 8, 3113\textendash{}3119.}{\normalsize \par}

\bibitem{key-Daigle} D. {\normalsize Daigle, }\emph{\normalsize A
necessary and sufficient condition for triangulability of derivation
of $k[X,Y,Z]$}{\normalsize , J. Pure Appl. Algebra 113 (1996), 297-305.} 

{\normalsize \bibitem{key-Poloni} P.-M. Poloni, }\emph{\normalsize Sur
les plongements des hypersurfaces de Danielewski}{\normalsize . These
de doctorat. Universite de Bourgogne (2008).}\end{thebibliography}
\end{document}